\documentclass[12pt]{article}
\usepackage[utf8]{inputenc}
\usepackage[a4paper, margin=1.3in]{geometry}
\usepackage[T1]{fontenc}
\usepackage{lmodern}
\usepackage{tikz-cd}
\usepackage{amsmath,amsfonts,amssymb}
\usepackage{amsthm}
\usepackage{graphicx}
\usepackage[all,cmtip]{xy}
\usepackage[english]{babel}
\usepackage{xurl}
\usepackage{comment}
\usepackage[x11names,dvipsnames,svgnames]{xcolor}
\usetikzlibrary{positioning}
\usetikzlibrary{calc}
\usetikzlibrary{patterns}

\newcommand{\keywords}[1]{%
\par\medskip
\noindent\textbf{Keywords: }#1
}
\newcommand{\subjclass}{%
\par\bigskip\par \noindent\textbf{2020 Mathematics Subject Classification: }%
}

\newtheorem{theorem}{Theorem}
\newtheorem{prop}[theorem]{Proposition}
\newtheorem{lemma}[theorem]{Lemma}
\newtheorem{cor}[theorem]{Corollary}
\theoremstyle{definition}
\newtheorem{definition}{Definition}
\newtheorem{example}{Example}
\newtheorem{question}{Question}
\newtheorem{problem}{Problem}
\newtheorem{construction}{Construction}

\newcommand\note[1]{\textbf{[#1]}}
\renewcommand\note[1]{}

\begin{document}
\title{\normalsize\bfseries
OBSTRUCTIONS TO
TOTAL RAINBOW FORESTS\\
IN EDGE-COLORED GRAPHS}
\vspace{-2em}
\author{\small
MARWA MOSALLAM\thanks{Department of Mathematics \& Statistics,
Binghamton University, Vestal, NY 13850, USA.
e-mail: \texttt{mmosallam@binghamton.edu}}
\text{ AND }
THOMAS ZASLAVSKY\thanks{Department of Mathematics \& Statistics,
Binghamton University, Vestal, NY 13850, USA.
e-mail: \texttt{zaslav@math.binghamton.edu}}
}

\date{}

\maketitle

\begin{abstract}
A total rainbow forest in an edge-colored graph is a forest that contains every edge color exactly once.  Using a necessary and sufficient condition that a total rainbow forest exists, we demonstrate the existence of huge numbers of edge-colored graphs that are minimal obstructions to such existence.  
\end{abstract}

\keywords{edge-colored graph; rainbow spanning tree; total rainbow forest}

\subjclass{05C15}.

\tableofcontents

\section{Introduction}

In an edge-colored graph, what is the maximum size of a rainbow (or ``heterochromatic'') forest; that is, a forest whose edges all have different colors?  Especially, does the graph have a \emph{total} rainbow forest, whose edges represent every color?  We study these questions and especially the obstructions to the existence of total rainbow forest.

Our attention was drawn to this question by Tarik Aougab, who proposed to the first author that a criterion for the existence of a total rainbow tree, or forest, or connected subgraph, might help with the problem of a largest 1-system of curves on a genus $g$ surface, and also told us about the use of the Matroid Intersection Theorem (which was suggested to him by Tony Huynh; see also \cite{Hu}).  For a statement of the 1-system problem we refer the reader to, e.g., \cite{AG, MRT}.

There has been considerable interest for years in rainbow forests, often on the existence of many edge-disjoint spanning rainbow forests in complete graphs; see for example \cite{Bru, KL, GK}.  Our focus is the existence and size of a single rainbow tree or forest, as for instance in \cite{Suz, AA}.

Let $G = (V,E)$ be a connected, edge-colored graph (not necessarily properly colored).  
Suzuki \cite[Theorem 7]{Suz} proved an important theorem, which we restate here in terms of matroids (although Suzuki stated and proved it purely in terms of graphs): If $G$ has order $n$ and any number $t$ of colors, then $G$ has a rainbow spanning tree if and only if for every edge set $A$ that is a union of color classes,
$$n - 1 - r_2 (E \setminus A) \leq r_1(A),$$
where $r_1$ and $r_2$ are the rank functions of matroids $M_1$ and $M_2$, both having the ground set $E,$ such that in $M_1$ a set is independent if it contains at most one edge of each color (this is known as the partition or transversal matroid; $r_1(B)$ is the number of color classes in $B$) and in $M_2$ a set is independent if it does not contain a cycle (this is the cycle matroid).  (See \cite{Ox} for matroids.)  We provide a complete proof of Suzuki's theorem generalized to disconnected graphs, using the Matroid Intersection Theorem; see Theorem \ref{main} .
As far as we know, the use of these matroids and that theorem  was introduced by Broersma and Li \cite[Section 2]{Bro}; they observed that a rainbow forest is an edge set that is independent in both the cycle matroid and the partition matroid and they used this fact  to determine the largest size of a rainbow forest \cite[Lemma 3]{Bro}.  (We learned of the Matroid Intersection approach from Tony Huynh, via Tarik Aougab.) 

In our study, usually $t=n-1$ (for connected $G$) so Suzuki's condition simplifies to $r_2(B) \geq r_1(B)$ for every set $B$ that is a union of color classes (Theorem \ref{compare}), which is our necessary and sufficient condition for existence of a total rainbow forest in terms of matroid ranks.  
 
This simpler condition suggests a new approach to the existence problem by way of forbidden subgraphs, which we call \emph{obstructions} to the existence of a rainbow spanning tree (Section \ref{obs}).  An obstruction is simply an edge-colored graph that has too many colors to contain a total rainbow forest.  We show that there are many and diverse obstructions of any order $n$ (barring trivially small $n$) that are minimal, in the sense of not having a smaller obstruction.  That is our principal contribution.

\section{The Rainbow Forest Inequality}

We assume throughout that \emph{$G = (V,E)$ is a simple graph of order $n$, edge colored with total number of colors $t$}.

The number of connected components of a graph, written $c(\ )$, is part of the cycle-matroid rank formula  $r_2(B) = n - c(V,B)$ for an edge set $B \subseteq E$.

We will use the concept of isomorphism of edge-colored graphs.  For us, isomorphic edge-colored graphs are not essentially different.  

\begin{definition}
Let $G_1$ and $G_2$ be edge-colored graphs with respective color sets $C_1$ and $C_2$ and edge-color functions $c_1: E_1 \to C_1$ and $c_2: E_2 \to C_2$.  An \emph{isomorphism} of $G_1$ and $G_2$, written $G_1 \cong G_2$, is a pair $\theta, f$ consisting of a graph isomorphism $\theta: G_1 \to G_2$ and a color-set bijection $f: C_1 \to C_2$ such that $c_2(\theta(e)) = f(c_1(e))$ for each edge $e \in E_1$.
\end{definition}

In other words, an isomorphism of edge-colored graphs $G$ and $G'$ is a graph isomorphism that preserves the partition of the edge set into color classes; it need not preserve the colors themselves.

Now we define the fundamental concept of our work.

\begin{definition}
    A \emph{rainbow} or \emph{heterochromatic subgraph} of an edge-colored graph $G$ is a subgraph such that no color appears more than once amongst the edges of the subgraph.
\end{definition}

We are especially interested in spanning trees or forests.  
If our forest need not contain the total number of colors $t$ but need only contain some of them, call $u$ the number of colors desired in the forest, where $u \leq t.$  This will be a total rainbow forest if $u=t$.  

\begin{definition}
    The \emph{Rainbow Forest Inequality} is defined as follows: 
    $$\forall A \subseteq E(G){:}\ r_1(A) + r_2(E \setminus A) \geq t.$$
Given a number of colors $u \leq t$, the \emph{Rainbow Forest Inequality for $u$} is defined as follows: 
    $$\forall A \subseteq E(G){:}\ r_1(A) + r_2(E \setminus A) \geq u.$$ 
\end{definition}

\begin{prop}\label{upperbd}
If $G$ satisfies the Rainbow Forest Inequality, the total number of colors $t$ has to satisfy  $t \leq n - c(G).$
\end{prop}

\begin{proof}
Set $A = \emptyset$ in the Rainbow Forest Inequality.
\end{proof}

\begin{definition}
    The \emph{Partial Rainbow Inequality} of an edge set $A$ is defined as follows: 
    $$ r_1(A) + r_2(E \setminus A) \geq t.$$
Given a number of colors $u \leq t$, the \emph{Partial Rainbow Inequality for $u$} of $A$ is defined as follows: 
    $$ r_1(A) + r_2(E \setminus A) \geq u.$$
\end{definition}

To prove Theorem \ref{main} below we need the previously defined matroids $M_1$ and $M_2$ on the edge set of $G$.  The connection is given by interpreting the Matroid Intersection Theorem.  We think of a matroid as a pair $(E,\mathcal{I})$ where $E$ is the ground set and $\mathcal{I}$ is the class of independent sets.

\begin{theorem}[{Edmonds' Matroid Intersection Theorem 
\cite[Theorem 11.3.15]{Ox}, \cite[Ch.\ 3, Exercise 40(c)]{GM}}] 
For any two matroids $M_1 = (E, \mathcal{I}_1)$ and $M_2 =(E,\mathcal {I}_2)$ we have
$$\max _{I\in {\mathcal {I}}_{1}\cap {\mathcal {I}}_{2}}|I|=\min _{A\subseteq E} [ r_{1}(A)+r_{2}(E\setminus A) ]$$
where $r_{1}$ and $r_{2}$ are the respective rank functions of $M_{1}$ and  $M_{2}$.
\end{theorem}

This gives our first theorem.  (A partial approach to this theorem is found in \cite[discussion after Lemma 3]{Bro}.)
\begin{theorem}\label{main}
Consider an edge-colored graph.  Let $t$ be the total number of colors on the edges.  
For a total rainbow forest to exist, it is necessary and sufficient that the Rainbow Forest Inequality be satisfied.

More generally, let $u$ be a number of colors such that $u \le t$.  For a rainbow forest of $u$ edges to exist, it is necessary and sufficient that the Rainbow Forest Inequality for $u$ be satisfied.
\end{theorem}

\begin{proof}
We use the partition matroid $M_1$ and the cycle matroid $M_2$.  In $M_1$ an edge set is independent when it has at most one edge of each color.  In $M_2$, an edge set is independent when it is a forest.  Thus, a common independent set of $M_1$ and $M_2$ is the same as a rainbow forest.  If the forest has $t$ edges, it is total; otherwise, it is missing at least one color.

The matroid intersection formula with our matroids $M_1$ and $M_2$ is
\begin{equation}
\max_{F :  \text{ rainbow forest}} |F| = \min_{A \subseteq E} \big ( r_1(A) + r_2(E \setminus A)  \big) .
\label{E:rainbow-mif}
\end{equation}
To have a rainbow forest of $u$ edges, $u$ must be no greater than the maximum; that is, the left side must be at least as large as $u$.  
That means the minimum on the right must be at least as large as $u$.  
Therefore, a $u$-edge rainbow forest exists if and only if the right side is at least as large as $u$ for every set $A$; that is,
$$r_1(A) + r_2(E \setminus A) \geq u$$
for every $A \subseteq E$.  This is the Rainbow Forest Inequality for $u$.

In particular, a total rainbow forest (where $u=t$) exists if and only if 
$$r_1(A) + r_2(E \setminus A) \geq t$$
for every set $A \subseteq E$.  This is the Rainbow Forest Inequality.
\end{proof}

\begin{definition}
    An \emph{absolute} or \emph{nonoverlapping} set is an edge set $A$ that does not have any common edge colors between it and $E \setminus A.$  That is, it is a union of color classes of edges.
    \end{definition}

Notice that on the right-hand side of Equation \eqref{E:rainbow-mif}, if $A$ intersects some color classes of $E \setminus A,$ we might as well force it to take the entire class for each of those colors, so it becomes absolute; that does not change $r_1(A)$ and it cannot make $r_2(E \setminus A)$ larger.  This justifies the following lemma.

\begin{lemma}\label{absA}
If the Partial Rainbow Inequality for $u$ is satisfied for all absolute $A,$ then it is satisfied for all $A$. 
\end{lemma}

\begin{proof}
Suppose $G$ satisfies the Partial Rainbow Inequality for $u$ for all absolute subgraphs, and consider an arbitrary subgraph $A$.  Let $A'$ be the subgraph $A$ together with all edges of $G$ that have the same color as an edge in $A$.  Then $r_1(A') = r_1(A)$ and  $A'$ is an absolute subgraph.  Thus, $r_1(A') + r_2(E \setminus A') \geq u$.  Because $E \setminus A \supseteq E \setminus A'$, we have $r_2(E \setminus A) \geq r_2(E \setminus A')$.  It follows that $A$ satisfies the Partial Rainbow Inequality for $u$.
\end{proof}

Now we come to the fundamental theorem.  In it, the essential point is that we eliminate mention of $A$ and only need to mention its complement, $B = E \setminus A$.  
Suzuki \cite[Theorem 7]{Suz} has an essentially equivalent theorem (stated and proved without matroids, and assuming a connected graph), but it is still stated with both sets $B$ and $A = E \setminus B$; as far as we know that is the first appearance of such a theorem.  Later, Vondr\'ak stated a version of Suzuki's theorem (without proof or citation) in terms of matroids \cite{Vondrak}.  We do not know of a previous publication of our statement or proof.

\begin{theorem}[Fundamental Theorem]\label{compare}
For an edge-colored graph $G$ with $n$ vertices and $t$ colors, $G$ has a total rainbow forest if and only if 
\begin{equation}
r_1(B) \leq r_2(B)
\label{fundt}
\end{equation}
for every absolute set $B \subseteq E(G).$ 

More generally, for any number $u$ such that $u \leq t$, $G$ has a rainbow forest with $u$ colors if and only if
\begin{equation}
t - u \leq r_2(B) - r_1(B)
\label{fundu}
\end{equation}
for every absolute set $B \subseteq E(G).$
\end{theorem}

\begin{proof}
For an absolute set $A$, $r_1(A) = t - r_1(E \setminus A)$.  Therefore, the Partial Rainbow Inequality for $u$ simplifies to $t - r_1(E \setminus A) + r_2(E \setminus A) \geq u$; applying Theorem \ref{main}  and replacing $E \setminus A$ by $B$, that implies the desired conclusion.
\end{proof}

We now have two criteria that are each equivalent to the existence of a total rainbow forest, or more generally a rainbow forest of specified size, in an edge-colored graph $G$.  Theorem \ref{main} is based on how how subgraphs of $G$ are related to the rest of $G$.  Theorem \ref{compare}, by contrast, concerns only an intrinsic property of the subgraph, regardless of the rest of the graph $G$.

\section{First Examples}\label{examples} 

\begin{prop}\label{12colors}
  All edge-colored graphs with one or two colors satisfy the Rainbow Forest Inequality.   
\end{prop}
  \begin{proof}
  Equivalently (by Theorems \ref{main} and \ref{compare}), we verify Equation \eqref{fundt} for absolute $B$.  We may assume $E \neq \emptyset$.  

We state Equation \eqref{fundt} for one color:  $1 = r_1(B) \leq r_2(B)$.  This is always satisfied because an edge exists.

 We state Equation \eqref{fundt} for two colors:  $r_1(B) \leq r_2(B)$.  This is always satisfied because either $|B| \geq 2$ so $r_2(B) \geq 2$, or else $|B| = 0$ so $r_1(B) = 0$, or $|B| = 1$ so $B$ is non-empty and $r_1(B) \geq 1$.
\end{proof}

\begin{example}\label{k3}
An edge-colored $K_3$ graph satisfies the Rainbow Forest Inequality if and only if it is nonrainbow.

For $K_3$, $t \leq |E| = 3$.  For the proof, if $t \leq 2$ apply the previous proposition.  If $t=3$, either $B \subset E(K_3)$ so $B$ is a forest, which obeys $r_1(B) \leq |B| = r_2(B)$, so $B$ satisfies Equation \eqref{fundt}, or $B = E(K_3)$ and $r_1(B) = t = 3 > r_2(B)$.
   \end{example}

\begin{prop}\label{3colors}
Every edge-colored graph $G$ with $t=3$ colors satisfies the Rainbow Forest Inequality, with the exception of a rainbow $K_3.$
\end{prop}
\begin{proof}
    We want to prove $r_1(B) \leq r_2(B)$ for every absolute set $B$.  We know $r_1(B) \leq 3$.  
    To violate the inequality, $r_2(B) \leq 2$ is necessary.  Thus, $B$ is either a forest of at most 2 edges, which cannot violate the inequality, or $B = E(K_3)$ and has 3 colors.  In that case, $B$ uses all three colors so, to be absolute, it must be $E(G)$. 
\end{proof}

\begin{example}
An edge-colored complete graph $K_4$ satisfies the Rainbow Forest Inequality if and only if $t < 4$.

For the proof, by Proposition \ref{upperbd}, $t \leq n - 1 = 3$, and that is covered by Propositions \ref{12colors} and \ref{3colors}.
\end{example}

\section{Obstruction Constructions}\label{obs}

Theorem \ref{compare} shows that there are subgraphs that prevent total rainbow forests.  
An ``obstruction'' is an edge-colored graph $B$ such that, if it is an absolute subgraph of some edge-colored graph $G$, it prevents $G$ from having a total rainbow forest.  Theorem \ref{compare} leads to the following precise definition.

\begin{definition}
    An \emph{obstruction} is an edge-colored graph $B$ that satisfies $r_1(B) > r_2(B)$.
    
    It is \emph{minimal} if no absolute proper subgraph, that is, no union of a proper subset of color classes, is an obstruction.
\end{definition}

In plain language, an obstruction is an edge-colored graph that has more colors than the largest size of a forest.  It has no total rainbow forest for the simplest reason. 

\begin{theorem}\label{r1r2min}
An obstruction $B$ is minimal if and only if every absolute proper subgraph $B' \subset B$ satisfies $r_1(B') \leq r_2(B')$.
\end{theorem}

\begin{proof}
An obstruction $B$ is minimal if and only if every absolute proper subgraph $B' \subset B$ is not an obstruction, i.e., satisfies $r_1(B') \leq r_2(B')$ (by Theorem \ref{compare}).
\end{proof}

\begin{prop}\label{forest}
A forest is not an obstruction.
\end{prop}

\begin{proof}
In a forest $F$, $r_1(F) \leq |E(F)| = r_2(F)$.
\end{proof}

\begin{prop}\label{obs12colors}
An obstruction must have at least $3$ colors.
\end{prop}

\begin{proof}
Suppose $B$ has fewer than 3 colors.  By Proposition \ref{12colors} it satisfies the Rainbow Forest Inequality.  By Theorem \ref{main} it has a total rainbow forest.  By Theorem \ref{compare} it satisfies $r_1(B) \leq r_2(B)$.  Thus, by definition, it is not an obstruction.
\end{proof}

\begin{theorem}\label{absminobs}
An edge-colored graph $G$ satisfies the Rainbow Forest Inequality if and only if no absolute subgraph is a minimal obstruction.
\end{theorem}

\begin{proof}
By Theorem \ref{compare}, $G$ has an absolute subgraph $B$ that is an obstruction if, and only if, it has a total rainbow forest.  If $B$ is not a minimal obstruction, it has an absolute proper subgraph $B'$ that is a minimal obstruction.  Then $B'$ is an absolute subgraph of $G$ and it contains no obstruction as an absolute proper subgraph, so it is a minimal obstruction in $G$.
\end{proof}

It follows immediately from Theorem \ref{compare} that:

\begin{theorem}\label{nominobs}
An edge-colored graph $G$ contains a total rainbow forest if and only if it contains no absolute subgraph that is a minimal obstruction.
\end{theorem}

This theorem points to the main question about total rainbow forests.

\begin{problem}
Find all minimal obstructions.
\end{problem}

This problem is certainly difficult, if not impossible.  There is a simple infinite family of minimal obstructions (Example \ref{cycle} below) but we also know construction methods that show minimal obstructions can be as complicated as any graphs (e.g., Construction \ref{rainbow-vertex}).
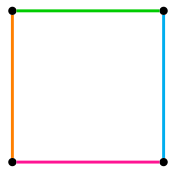
\begin{figure}[h]

\[\begin{tikzpicture}[
    every path/.style={line width= 1.1 pt},
    every node/.style={circle,fill=black,draw,inner sep=0pt,minimum size=2pt}
]

\node (A) at (0,0) {};
\node (B) at (2,0) {};
\node (C) at (2,2) {};
\node (D) at (0,2) {};

\draw[DeepPink] (A)--(B);
\draw[cyan] (B)--(C);
\draw[green!80!black] (C)--(D);
\draw[orange] (D)--(A);
\end{tikzpicture}\] 
\caption{A typical rainbow cycle obstruction (Example \ref{cycle}).}
\label{rainbowcycle}
\end{figure}
\begin{example}\label{cycle}
Any rainbow cycle is a minimal obstruction.

Moreover, a rainbow 3-cycle $C_3$ is the only obstruction of order 3 (see Example \ref{k3}).

\end{example}

\begin{example}\label{k4-original}
The following specific $K_4$ with $t = 4$ is an obstruction that is not rainbow.  It has $r_1 = 4 > r_2 = 3$.  

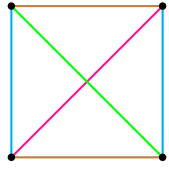
\begin{figure}[h]
\[\begin{tikzpicture}[
        every path/.style={thick}, 
        every node/.style={circle,fill=black,draw,inner sep = 0pt, minimum size= 2pt}
        ]
    \draw[brown] (0,0)--(2,0);
    \draw[cyan] (2,0)--(2,2);
    \draw[brown] (2,2)--(0,2);
    \draw[cyan] (0,2)--(0,0);
    \draw[DeepPink] (0,0)--(2,2);
    \draw[green] (0,2)--(2,0);
    \foreach \corner in {(0,0), (2,0), (2,2), (0,2)} 
        \node at \corner {};
\end{tikzpicture} \]

\caption{The non-rainbow minimal obstruction in Example \ref{k4-original}.  It is one of the six minimal obstructions on $K_4$ (see Figure \ref{F-allK4}).}
\end{figure}

We prove it is a minimal obstruction.
First, it contains no absolute rainbow cycle (although it does contain rainbow 3-cycles).
Second, deleting any one color class does not reduce $r_2$ while deleting any two color classes reduces $r_2$ by at most 1.
Finally, deleting three color classes gives a forest, which is not an obstruction (Proposition \ref{forest}).
\end{example}

\begin{prop}\label{obs-order}
A minimal obstruction $B$ has order at least $3$.  
\end{prop}

\begin{proof}
An obstruction must have an edge.  If it had only one edge, then $r_1(B) = 1 = r_2(B)$ and it would not be an obstruction.  
\end{proof}

\begin{prop}\label{obs-ranks}
For a minimal obstruction $B$, let $E_i = \{ e \in E(B): e \text{ has color } i\}$.  Then $r_1(B) - 1 = r_2(B) = r_2(B \setminus E_{j})$ for every color $j$ in $B$.
\end{prop}

\begin{proof}
By minimality of $B$ and Theorem \ref{r1r2min}, every collection $K$ of colors in $B$ with one or more missing (that is, with $|K| < r_1(B)$) satisfies
$|K| = r_1(\bigcup_{i\in K} E_i) \leq r_2(\bigcup_{i\in K} E_i)$.  In particular, for every color $j$, 
$$r_1(B) - 1 \leq r_2(\bigcup_{i \neq j} E_i) \leq r_2(B).$$ 
Because $B$ is an obstruction, $r_1(B) > r_2(B) \geq r_1(B) - 1$, so 
$r_2(B) = r_1(B) - 1$, which both equal $r_2(\bigcup_{i \neq j} E_i) = r_2(B \setminus E_{j}).$
\end{proof}

A \emph{bond} of $G$ is a minimal edge set whose deletion increases the number of components of $G$; that is, deleting it decreases $r_2$ by exactly 1.  For instance, an isthmus is a bond.

\begin{lemma}\label{bond-lemma}
If $B$ is an obstruction, then deleting a color class that contains a bond gives an obstruction contained in $B$.
\end{lemma}

\begin{proof}
By Proposition \ref{obs12colors}, $B$ must have at least three color classes.  Deleting the color class $C$ that contains a bond reduces $r_2$ by at least 1 because deleting the bond reduces $r_2$ by 1.  Since deleting a color class reduces $r_1$ by exactly 1, we still have $r_1 > r_2$ after deletion.
\end{proof}

From this lemma we obtain a property of minimal obstructions.  

\begin{prop}\label{no-bond}
A minimal obstruction cannot have a color class that contains a bond.
\end{prop}

\begin{proof}
If it did, then by Lemma \ref{bond-lemma} deleting that color class would create an obstruction, contradicting minimality.
\end{proof}

Examples \ref{K4-obs} and \ref{diamond-obs} show all minimal obstructions of order 4 other than the rainbow 4-cycle.

\begin{example}\label{K4-obs}
Figure \ref{F-allK4} shows all six minimal obstructions whose underlying graph is $K_4$.

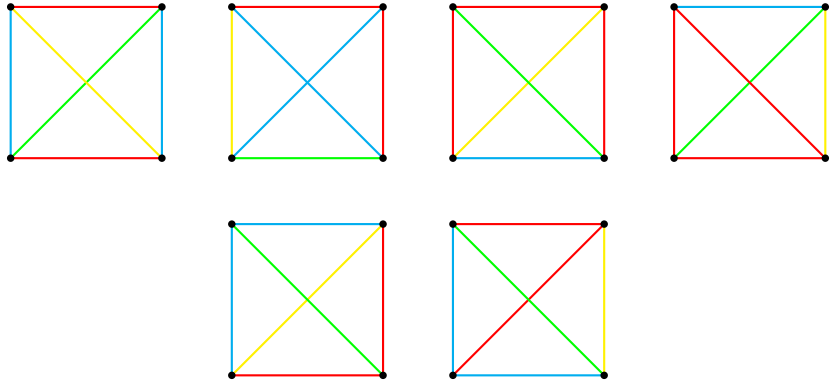
\begin{figure}[htb]
  \[\begin{tikzpicture}[
        every path/.style={thick}, 
        every node/.style={circle,fill=black,draw,inner sep = 0pt, minimum size= 2pt}
        ]
    \draw[red] (0,0)--(2,0);
    \draw[cyan] (2,0)--(2,2);
    \draw[red] (2,2)--(0,2);
    \draw[cyan] (0,2)--(0,0);
    \draw[green] (0,0)--(2,2);
    \draw[yellow] (0,2)--(2,0);
    \foreach \corner in {(0,0), (2,0), (2,2), (0,2)} 
        \node at \corner {};
\end{tikzpicture} \quad \quad \begin{tikzpicture}[
        every path/.style={thick}, 
        every node/.style={circle,fill=black,draw,inner sep = 0pt, minimum size= 2pt}
        ]
    \draw[green] (0,0)--(2,0);
    \draw[red] (2,0)--(2,2);
    \draw[red] (2,2)--(0,2);
    \draw[yellow] (0,2)--(0,0);
    \draw[cyan] (0,0)--(2,2);
    \draw[cyan] (0,2)--(2,0);
    \foreach \corner in {(0,0), (2,0), (2,2), (0,2)} 
        \node at \corner {};
\end{tikzpicture} \quad \quad \begin{tikzpicture}[
        every path/.style={thick}, 
        every node/.style={circle,fill=black,draw,inner sep = 0pt, minimum size= 2pt}
        ]
    \draw[cyan] (0,0)--(2,0);
    \draw[red] (2,0)--(2,2);
    \draw[red] (2,2)--(0,2);
    \draw[red] (0,2)--(0,0);
    \draw[yellow] (0,0)--(2,2);
    \draw[green] (0,2)--(2,0);
    \foreach \corner in {(0,0), (2,0), (2,2), (0,2)} 
        \node at \corner {};
\end{tikzpicture} \quad \quad \begin{tikzpicture}[
        every path/.style={thick}, 
        every node/.style={circle,fill=black,draw,inner sep = 0pt, minimum size= 2pt}
        ]
    \draw[red] (0,0)--(2,0);
    \draw[yellow] (2,0)--(2,2);
    \draw[cyan] (2,2)--(0,2);
    \draw[red] (0,2)--(0,0);
    \draw[green] (0,0)--(2,2);
    \draw[red] (0,2)--(2,0);
    \foreach \corner in {(0,0), (2,0), (2,2), (0,2)} 
        \node at \corner {};
\end{tikzpicture}\]

\[\begin{tikzpicture}[
        every path/.style={thick}, 
        every node/.style={circle,fill=black,draw,inner sep = 0pt, minimum size= 2pt}
        ]
    \draw[red] (0,0)--(2,0);
    \draw[red] (2,0)--(2,2);
    \draw[cyan] (2,2)--(0,2);
    \draw[cyan] (0,2)--(0,0);
    \draw[yellow] (0,0)--(2,2);
    \draw[green] (0,2)--(2,0);
    \foreach \corner in {(0,0), (2,0), (2,2), (0,2)} 
        \node at \corner {};
\end{tikzpicture}\quad \quad \begin{tikzpicture}[
        every path/.style={thick}, 
        every node/.style={circle,fill=black,draw,inner sep = 0pt, minimum size= 2pt}
        ]
    \draw[cyan] (0,0)--(2,0);
    \draw[yellow] (2,0)--(2,2);
    \draw[red] (2,2)--(0,2);
    \draw[cyan] (0,2)--(0,0);
    \draw[red] (0,0)--(2,2);
    \draw[green] (0,2)--(2,0);
    \foreach \corner in {(0,0), (2,0), (2,2), (0,2)} 
        \node at \corner {};
\end{tikzpicture}\]
\caption{The six minimal obstructions based on $K_4$. }
\label{F-allK4}
\end{figure}

To prove the list is complete, consider the sizes of the four color classes, $E_1, E_2, E_3, E_4$.  They may be either $3,1,1,1$ or $2,2,1,1$, respectively.  We find all the minimal obstructions for each of these types.

In the former type, if the three edges of $E_1$ are all the edges incident to a vertex, the other colors form an absolute rainbow $K_3$ subgraph, which is already an obstruction.  If the edges are a path or a cycle, we get graphs (1, 3) and (1, 4) in Figure \ref{F-allK4} (that is, in positions 3 and 4 of row 1).

Each of these graphs is a minimal obstruction because it does not contain as an absolute subgraph a rainbow $K_3$, which is the only smaller obstruction (see Example \ref{cycle}).
\end{example}

\begin{example}\label{diamond-obs}
Consider a graph $B$ that is a diamond colored to be a minimal obstruction (Figure \ref{F-alldiamond}.  Thus, $r_1(B) = r_2(B) + 1 = 4$.  The five edges colored in four colors must have one doubled color.  The two triangles must each have a doubled color, or $B$ would not be minimal.  There are three ways to arrange this:  with the doubled color (1) on the common edge of the triangles, (2) on adjacent edges in different triangles, and (3) on nonadjacent edges of the two triangles.  Each of these is clearly a minimal obstruction. 

\begin{figure}[htbp]
\begin{center}

\begin{tikzpicture}[
        every path/.style={thick}, 
        every node/.style={circle,fill=black,draw,inner sep = 0pt, minimum size= 2pt}
        ]
    \draw[red] (0,0)--(2,0);
    \draw[cyan] (2,0)--(2,2);
    \draw[yellow] (2,2)--(0,2);
    \draw[red] (0,2)--(0,0);
    \draw[green] (0,0)--(2,2);
    \foreach \corner in {(0,0), (2,0), (2,2), (0,2)} 
        \node at \corner {};
\end{tikzpicture} \quad \quad \begin{tikzpicture}[  every path/.style={thick}, 
        every node/.style={circle,fill=black,draw,inner sep = 0pt, minimum size= 2pt}
        ]
    \draw[red] (0,0)--(2,0);
    \draw[cyan] (2,0)--(2,2);
    \draw[yellow] (2,2)--(0,2);
    \draw[green] (0,2)--(0,0);
    \draw[red] (0,0)--(2,2);
    \foreach \corner in {(0,0), (2,0), (2,2), (0,2)} 
        \node at \corner {};
\end{tikzpicture} \quad \quad \begin{tikzpicture}[  every path/.style={thick}, 
        every node/.style={circle,fill=black,draw,inner sep = 0pt, minimum size= 2pt}
        ]
    \draw[cyan] (0,0)--(2,0);
    \draw[red] (2,0)--(2,2);
    \draw[yellow] (2,2)--(0,2);
    \draw[red] (0,2)--(0,0);
    \draw[green] (0,0)--(2,2);
    \foreach \corner in {(0,0), (2,0), (2,2), (0,2)} 
        \node at \corner {};
\end{tikzpicture}
\caption{The three minimal obstructions based on the diamond graph $K_4 \setminus e$.}
\label{F-alldiamond}
\end{center}
\end{figure}
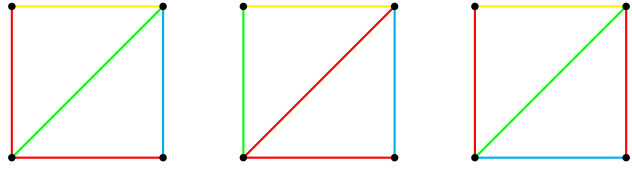
\end{example}
We have completed the list of minimal obstructions of order 4.

\begin{prop}\label{order4}
The graphs in Figures \ref{F-allK4} and \ref{F-alldiamond} and the rainbow $C_4$ constitute the complete set of minimal obstructions of order $4$.
\end{prop}

\begin{proof}
We consider all graphs $G$ of order 4 that might be obstructions, in turn.  If the graph is a forest it cannot be an obstruction (Proposition \ref{forest}); if it is disconnected and not a forest, it must be a rainbow triangle plus an isolated vertex, which is not minimal.  Therefore, $G$ is connected with at least 4 edges.

Those graphs are $K_4$ (6 edges), the diamond graph (5 edges), and graphs with 4 edges.  The former two are solved in Examples \ref{K4-obs} and \ref{diamond-obs}.  Considering 4 edges, an obstruction must have 4 colors, so it is rainbow.  If it contains a triangle, it is not minimal because the triangle is rainbow.  There remains only a rainbow $C_4$, which is a minimal obstruction.  
\end{proof}

\subsection{Rainbow Vertex Obstruction}

If an obstruction of order $n$, colored by $n$ colors, has a \emph{rainbow vertex}, which means a vertex of valency $n-1$ with no color repeated on its incident edges, then the remainder can be almost any graph.  That seems to dash hope of classifying all obstructions.

\begin{construction}\label{rainbow-vertex}
   Let $G_0$ be a graph on vertices $v_1,\ldots,v_{n-1}$ without isolated vertices.  Color its edges with color $b$ (for blue).  Add a vertex $v_n$ and all $n-1$ edges $v_iv_n$ with color $i$ (different from $b$) on edge $v_iv_n$.  Call this graph $G$.  Also, if desired, add any ``extra'' edges where $G_0$ has non-adjacencies, colored with any colors from $1,...,n-1$.  Call this graph $G_{\text{rv}}$.  We call it a \emph{rainbow vertex graph} since $v_n$ is a rainbow vertex. 
\end{construction}

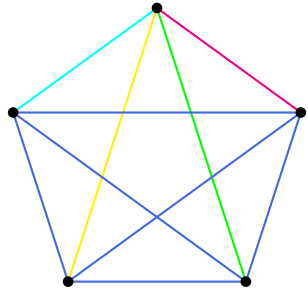
\begin{figure}[htbp]
\begin{center}
  \begin{tikzpicture}[  every path/.style={thick}, 
        every node/.style={circle,fill=black,draw,inner sep = 0pt, minimum size= 2pt}
        ]
    \foreach \i in {1,...,5}{
        \path (\i*360/5+18:2) coordinate (n\i);
        \ifnum\i>1\relax
            foreach \j in {\i,...,1}{ (n\i) edge (n\j) } 
        \fi;
    }
     \draw[Cyan] 
    (n1) -- (n2);
    \draw[yellow] 
    (n1) -- (n3);
    \draw[green] 
    (n1) -- (n4);
    \draw[magenta] 
    (n1) -- (n5);
    \draw[RoyalBlue] 
    (n2) -- (n3) (n2) -- (n4)(n2) -- (n5)(n3) -- (n4) (n3) -- (n5)(n4) -- (n5);
    \foreach \i in {1,...,5}{\fill (n\i) circle(2pt);}
  \end{tikzpicture}
\caption{An example of a rainbow-vertex complete graph. The whole graph is the graph $G,$ the blue part of the graph is the graph $G_0$ and the top vertex is the rainbow vertex $v_n$ where $n = 5$ in this example. In our example here $G_0$ has no non-adjacencies so no extra edges were added. }
\label{rainbow-vertex-k5}
\end{center}
\end{figure}

\begin{theorem}\label{T-rainbow-vertex}
$G_{\mathrm{rv}}$ is a minimal obstruction for $n\geq3.$ 
\end{theorem}

\begin{proof}
Since $r_1(G_{\text{rv}}) = n$ and $r_2(G_{\text{rv}}) = n-1$, $G_{\text{rv}}$ is an obstruction.

Case 1.  Assume no extra edges were added.  When deleting some color classes to get an absolute subgraph of $G$, we may keep or discard the color class of $b$.  
Deleting color class $b$ leaves a rainbow star, which satisfies the Rainbow Forest Inequality as it is a rainbow tree, or by Proposition \ref{forest}.

If we keep color class $b$, we first remove one edge at the rainbow vertex $v_n$; this gives $r_1=n-1=r_2$.  Now we remove the other edges incident to $v_n$ (each of which is a color class), one at a time.  Each time we remove such an edge, the number of components increases by at most 1, so $r_2$ decreases by at most 1.  If we remove a total of $s$ edges after the first, we reduce $r_1$ to $n-1-s$ and $r_2$ to, at minimum, $n-1-s$, so we end up with $r_1 \leq r_2$.  Thus, for any $s \geq 0$, deleting $1+s$ of the edges incident to $v_n$ gives a graph that is not an obstruction.

Case 2.  Suppose extra edges may be added.  
Let $G_{\text{rv}}'$ be the graph without extra edges, that is, it is $G_0$ together with all the edges incident with $v_n$.  

Consider deleting from $G_{\text{rv}}$ a nonempty union $U$ of some of its color classes.  
Then $G_{\text{rv}}' \setminus U \subseteq G_{\text{rv}} \setminus U$, so $r_2(G_{\text{rv}} \setminus U) \geq r_2(G_{\text{rv}}' \setminus U)$.  Also, $r_2(G_{\text{rv}}' \setminus U) \geq r_1(G_{\text{rv}}' \setminus U)$ by the previous case and $r_1(G_{\text{rv}}' \setminus U) = r_1(G_{\text{rv}} \setminus U)$ by construction.  Therefore, $r_2(G_{\text{rv}} \setminus U) \geq r_1(G_{\text{rv}} \setminus U)$, so $G_{\text{rv}} \setminus U$ is not an obstruction.  This proves that $G_{\text{rv}}$ is minimal.
\end{proof}

\begin{prop}\label{rainbow-isom}
Construct two rainbow vertex graphs $G$ and $G'$ of the same order $n$ as in Construction \ref{rainbow-vertex}, without extra edges.  Then $G \cong G'$ as edge-colored graphs if and only if $G_0 \cong G_0'$ as graphs.
\end{prop}

\begin{proof}
The reverse direction is obvious.  Assume $G \cong G'$ by an isomorphism $\theta$.  The rainbow vertices $v\in V$ and $v' \in V'$ must correspond under $\theta$; that is, $\theta(v) = v'$.  Then $G_0$ is isomorphic to $G_0'$ under $\theta$.
\end{proof}

\begin{cor}
A strict lower bound for the number of nonisomorphic minimal obstructions of order $n \geq 4$ is the number of isomorphism types of simple graphs of order $n-1$ without isolated vertices.
\end{cor}
\begin{proof}
Proposition \ref{rainbow-isom} implies that the number of isomorphism types of rainbow vertex graphs $G$ of order $n$ equals the number of isomorphism types of simple graphs $G_0$ of order $n-1$ without isolated vertices.
Each rainbow vertex graph of order $n$ is a minimal obstruction.
Examples that are not rainbow vertex graphs, such as a rainbow $n$-cycle, demonstrate that the bound is strict. 
\end{proof}

This corollary shows that the number of minimal obstructions of each order is quadratically exponential.
Indeed, there is the following asymptotic formula for the number of unlabeled simple graphs of order $n-1$:  $2^{\binom{n-1}{2}}/(n-1)!$ (\cite{HP} as cited in \cite[Sequence A000088]{OEIS}; see the first formula there for $a(n)$).  The number without isolated vertices is asymptotically the same.  Using Stirling's approximation, that is asymptotically
$$
\frac{2^{(n-1)(n-2)/2}}{((n-1)/e)^{n-1}\sqrt{2(n-1)\pi}} = \bigg(\frac{e \sqrt{2}^{\,n-2}}{n-1}\bigg)^{n-1} \frac{1}{\sqrt{2(n-1)\pi}}.
$$
This suggests there are so many nonisomorphic obstructions of rainbow-vertex type that searching for obstructions in a particular large edge-colored graph is not feasible.  Even those with maximum valency $d$ that might be absolute subgraphs of $d$-regular graphs are very numerous.  Indeed, minimal obstructions of order $n$  are very numerous, as one can infer from Proposition \ref{rainbow-isom}, and by Proposition \ref{T-rainbow-vertex-complete} those that are complete graphs are about as numerous.  

Despite this, detecting rainbow vertex subgraphs of order $d+1$ (of the kind without extra edges) in a regular edge-colored graph $G$ of order $n$ and valency $d$ is not difficult, since all rainbow vertices $v$ in $G$ can be detected in time $O(d^2n)$ by testing each vertex $v$ of $G$ for having all $d$ colors (time $O(d)$) and, if it does have them, testing the induced subgraph on $N(v)$ for having only one color and no isolated vertices (time $O(d^2)$).

We add a further detail about isomorphism when there are extra edges.  It requires $G_0$ to have minimum valency $>1$.  We write $G_{\text{rv}}$ to emphasize that now we allow extra edges in Construction \ref{rainbow-vertex}.

\begin{prop}\label{T-rainbow-vertex-complete}
In Construction \ref{rainbow-vertex}, if $G_0$ and $G_0'$ have minimum valency at least 2, and if $G_{\mathrm{rv}}$ and $G'_{\mathrm{rv}}$ are isomorphic, then $G_0$ and $G_0'$ are isomorphic graphs.
\end{prop}

\begin{proof}
The valency constraint ensures there is only one rainbow vertex in any $G_{\text{rv}}$.  It follows that the color class $b$ is intrinsic to $G_{\text{rv}}$, and the proof of Proposition \ref{rainbow-isom} goes through.
\end{proof}

\begin{figure}[htbp]

\[\begin{tikzpicture}
\filldraw[fill=cyan!7, draw=cyan, thick]
    (0,0) ellipse (5 cm and 2.5cm);
\node[black] at (0.5,1.9) {$G_0 \cong G_0'$};

\node[circle,fill=red,inner sep=1.8pt] (v1) at (-1.1, 1.4) {};
\node[circle,fill=red,inner sep=1.8pt] (v2) at ( 1.3, 1.25) {};
\node[circle,fill=red,inner sep=1.8pt] (v3) at (-2.6,-1.3) {};
\node[circle,fill=red,inner sep=1.8pt] (v4) at ( 0.8,-1.4) {};
\node[circle,fill=red,inner sep=1.8pt] (v5) at ( 0.2, 0.2) {};
\node[above left,text=red] at (v1) {$a$};
\node[below left,text=red] at (v3) {$b$};

\node[circle,fill=red,inner sep=1.8pt] (u) at (6.5,0) {};
\node[text=red,right] at (u) {$u$};

\tikzset{
    edge/.style={line width= 2 pt}
}
\draw[red, thick]
    (v1) --
    node[pos=0.3,left]  { \color{black}$ 1$ in $G_{rv}'$}
    node[pos=0.7,right] { \color{black}$2$ in $G_{rv} $}
    (v3);
\draw[red] (u) -- node[pos=.5,above right,text= black] {$1$} (v1);
\draw[red] (u) -- node[pos=.7,below left,text= black]       {$2$} (v2);
\draw[red] (u) -- node[pos=.6,above,text= black]        {$4$} (v3);
\draw[red] (u) -- node[pos=.5,below,text=black]       {$5$} (v4);
\draw[red] (u) -- node[pos=.6,below left,text= black]  {$3$} (v5);

\end{tikzpicture}\]

\caption{An illustration of the failure of the converse of Proposition \ref{T-rainbow-vertex-complete}.  The edge $ab$ has different colors, 2 or 1 depending on whether it is in $G_{rv}$ or $G_{rv}^{'}$.}
\label{rainbow-noniso}
\end{figure}
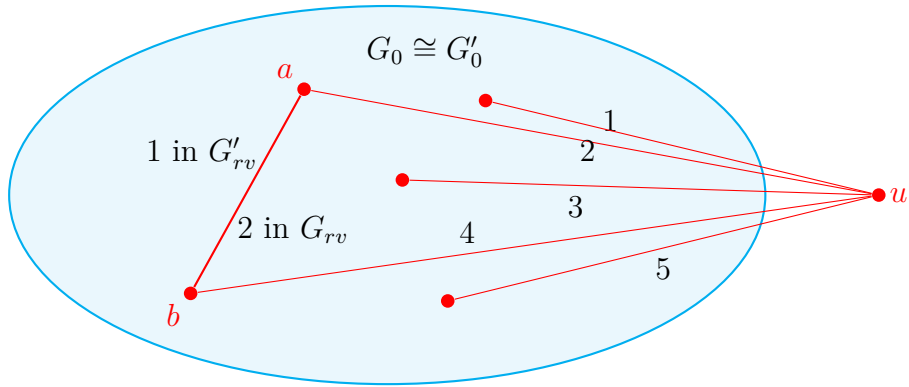
The converse is not true.  The extra edges in $G_{\mathrm{rv}}$ and $G'_{\mathrm{rv}}$ may be different, and even if they are the same edges, they may not have colors that correspond under the colored-graph isomorphism of $G$ and $G'$ as in Proposition \ref{rainbow-isom}. 
See Figure \ref{rainbow-noniso} for an example.
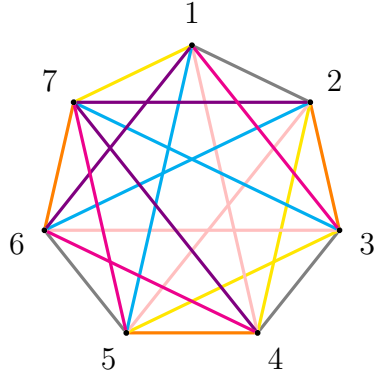
\begin{figure}[h]
\begin{center}
\begin{tikzpicture}[
    scale=1.0,
    every node/.style={
        circle,
        fill=black,
        draw,
        inner sep= 0pt,
        label distance=2mm},
    every path/.style={line width=1.2pt}
]

\foreach \i/\a in {
1/90,
2/38.57,
3/-12.86,
4/-64.29,
5/-115.71,
6/-167.14,
7/141.43}
{
    \node[label=\a:\i] (v\i) at (\a:2) {};
}


\draw[Grey] (v1)--(v2);
\draw[Grey] (v3)--(v4);
\draw[Grey] (v5)--(v6);


\draw[orange] (v2)--(v3);
\draw[orange] (v4)--(v5);
\draw[orange] (v6)--(v7);


\draw[yellow!90!orange] (v3)--(v5);
\draw[yellow!90!orange] (v2)--(v4);
\draw[yellow!90!orange] (v1)--(v7);


\draw[pink] (v1)--(v4);
\draw[pink] (v2)--(v5);
\draw[pink] (v3)--(v6);


\draw[cyan] (v1)--(v5);
\draw[cyan] (v2)--(v6);
\draw[cyan] (v3)--(v7);


\draw[violet] (v1)--(v6);
\draw[violet] (v2)--(v7);
\draw[violet] (v4)--(v7);


\draw[magenta] (v1)--(v3);
\draw[magenta] (v4)--(v6);
\draw[magenta] (v5)--(v7);

\end{tikzpicture}
\end{center}
\caption{An illustration of Pure Equinumerosity for $K_7$ .}
\label{pureequi}
\end{figure}

\subsection{Equinumerosity}

There is an interesting type of minimal obstruction based on the concept of equinumerosity, in which all color classes have about the same number of edges.
Our analysis is guided by the partition and cycle matroids and Theorem \ref{r1r2min}.

\begin{theorem}[Pure Equinumerosity]\label{equinum-odd}
For odd $n$, if $K_n$ is colored in $n$ colors so that every color class has the same number of edges, then this graph is a minimal obstruction.
\end{theorem}

We call such a graph a \emph{pure equinumerous obstruction}.

\begin{proof}
Because the graph has $n$ colors, it is an obstruction.  We show that no proper absolute subgraph is an obstruction.  

The number of edges in each color class is $\frac12(n-1)$.  An absolute subgraph $H$ with $r$ colors has $|E(H)| = \frac12(n-1)r$ edges.  It is an obstruction if and only if $r = r_1(H) > r_2(H) = n - c(H)$; that is, $c(H) > n-r$.  A graph $H$ of order $n$ with $c(H) > n-r$ has the most edges when it has the fewest components (which in this case is $n-r+1$ components) and one component is a complete graph of order $r$ while all other components are isolated vertices, which gives it at most $\frac12r(r-1)$ edges. 
Thus, $\frac12(n-1)r = |E(H)| \leq \frac12r(r-1)$.  
It follows that $n-1 \leq r-1$.  In other words, if $H$ is an obstruction, then $r=n$, i.e., $H = K_n$, which proves minimality of $K_n$. 
\end{proof}

We want similar results for even $n$, but since each color class should have the fractional number $\frac12(n-1)$ of edges, we have to relax the requirement of equinumerosity.  We call a coloring \emph{nearly equinumerous} if the numbers of edges in different color classes differ by at most 1.
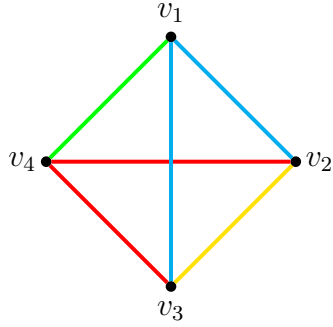
\begin{figure}
    \centering
    \[
\begin{tikzpicture}[
    scale=1.1,
    every node/.style={
        circle,
        fill=black,
        draw,
        inner sep=0pt,
        minimum size=2.5pt
    },
    every path/.style={line width=1.5pt}
]

\node[label=above:$v_1$] (v1) at (90:1.5) {};
\node[label=right:$v_2$] (v2) at (0:1.5) {};
\node[label=below:$v_3$] (v3) at (-90:1.5) {};
\node[label=left:$v_4$] (v4) at (180:1.5) {};

\draw[red] (v4)--(v2);
\draw[red] (v3)--(v4);

\draw[cyan] (v1)--(v3);
\draw[cyan] (v1)--(v2);

\draw[green] (v1)--(v4);

\draw[yellow!85!orange] (v2)--(v3);

\end{tikzpicture}
\]
    \caption{An illustration of Near Equinumerosity for $K_4$.}
    \label{NearEqui}
\end{figure}

\begin{theorem}[Near Equinumerosity]\label{equinum-even}
For even $n \geq 4$, if $K_n$ is colored in $n$ colors so that all color classes have numbers of edges that differ by at most $1$, then this graph is a minimal obstruction.
\end{theorem}

We call such a graph a \emph{nearly equinumerous obstruction}.

\begin{proof}
Near equinumerosity implies that $n/2$ color classes have $\frac12(n-2)$ edges each and $n/2$ have $\frac12 n$ edges each.  As in the previous proof we suppose an absolute subgraph $H$ with $r$ colors is an obstruction so $c(H) > n-r$.  Then $|E(H)| \leq \frac12r(r-1)$ .  The fewest edges possible in $H$ is $\frac12(n-2)r$, so to have an obstruction we need $\frac12(n-2)r \leq \frac12r(r-1)$, which implies that $n-1 \leq r$.  

Suppose $r=n-1$.  Then $H$ has at least $\frac12n(n-1) - \frac12 n$ edges, while $c(H) \geq 2$.  The most edges possible in $H$ with 2 components is $\frac12(n-1)(n-2)$.  If $H$ is an obstruction, $\frac12(n-2)(n-1) \leq |E(H)| \leq \frac12(n-1)(n-2)$, which is possible only if $H = K_{n-1} \cup K_1$ and every one of the $n-1$ color classes in $H$ has the smaller number of edges, $\frac12(n-2)$.  This is possible only if $n-1 \leq \frac12n$, that is, $n \leq 2$, which contradicts our assumption.  Therefore, the only obstruction that is an absolute subgraph of this colored $K_n$ is the entire graph, and it is minimal.
\end{proof}

\subsection{Bicolored Vertex}

Suzuki proved the following second theorem:

\begin{theorem}[{\cite[Theorem 8]{Suz}}]\label{suz}  An edge-colored complete graph $K_n$ has a rainbow spanning tree if every color class has at most $n/2$ edges.
\end{theorem}

Suzuki's theorem implies there are at least $n-1$ colors.  This theorem was improved by Akbari and Alipour.  In a coloring by $k$ colors, let $a_i$ denote the number of edges with color $i$.

\begin{theorem}[{\cite[Theorem 1]{AA}}]\label{aa}
Let $K_n$ be edge-colored in $k \geq n-1$ colors so that $1 \leq a_1 \leq \cdots \leq a_k \leq \frac12(n+3)$.  Then there is a rainbow spanning tree.
\end{theorem}

We use these results in Theorems \ref{valvert-odd} and \ref{valvert-even} to establish a new family of obstructions based on a partial kind of equinumerosity.  A \emph{bicolored vertex} is a vertex whose incident edges are colored in two colors that do not appear on any other edges.

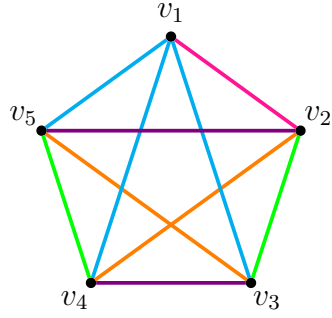
\begin{figure}
    \centering
\[
\begin{tikzpicture}[
    scale=1.0,
    every path/.style={line width=1.4pt},
    every node/.style={
        circle,
        fill=black,
        draw,
        inner sep=0pt,
        minimum size=2.5pt
    }
]

\foreach \i/\a in {
1/90,
2/18,
3/-54,
4/-126,
5/162}
{
    \node[label=\a:$v_{\i}$] (v\i) at (\a:1.8) {};
}

\draw[DeepPink] (v1)--(v2);
\draw[violet] (v3)--(v4);

\draw[green] (v2)--(v3);
\draw[green] (v4)--(v5);

\draw[orange] (v3)--(v5);
\draw[cyan] (v1)--(v4);

\draw[cyan] (v1)--(v5);
\draw[orange] (v2)--(v4);

\draw[cyan] (v1)--(v3);
\draw[violet] (v2)--(v5);

\end{tikzpicture}
\]
    \caption{An illustration of Theorem \ref{valvert-odd} for $n=5$ with $v_1$ the bicolored vertex.}
    \label{bivertexodd}
\end{figure}
 
\begin{theorem}[Bicolored Vertex: Odd Case]\label{valvert-odd}
For odd $n\geq3$, let $K_n$ have $n$ color classes such that one vertex $v$ is incident with all the edges of two classes but no other edges, and the remaining edges are colored in $n-2$ other colors with exactly $\frac12(n-1)$ edges in each class.  Then this graph is a minimal obstruction.
\end{theorem}

\begin{proof}
We consider ways to construct absolute proper subgraphs.  Let the color classes incident with $v,$ the bicolored vertex, be $C_n$ and $C_{n-1}$ (so $|C_n|+|C_{n-1}|=n-1$. 
First, delete $C_n \cup C_{n-1}$, leaving $K_n \setminus v = K_{n-1}$ colored so every color class has size $\frac12(n-1)$.  By Suzuki's theorem, this graph has a rainbow spanning tree, which is total because in $K_n \setminus v$ we are using $n-2$ colors; thus, $K_n \setminus v$ contains no obstruction, by Theorem \ref{compare}.  
Second, delete $C_n$ but not $C_{n-1}$ (or the reverse).  By Proposition \ref{no-bond}, this graph cannot be an obstruction, since $C_{n-1}$ is a color class that is a bond, and deleting it gives $K_n \setminus v$ which by Suzuki's theorem, as before, cannot be or contain an obstruction.  
Third, deleting any color classes except $C_{n-1}$ and $C_n$ leaves a graph which still has the spanning tree $C_n \cup C_{n-1}$, which has $r_2 = n-1 \geq r_1$ so it is not an obstruction.
\end{proof}

There is not a generalization that allows arbitrary coloring of $K_n \setminus v$.  For an example, see Figure \ref{valvert-odd-bad}.

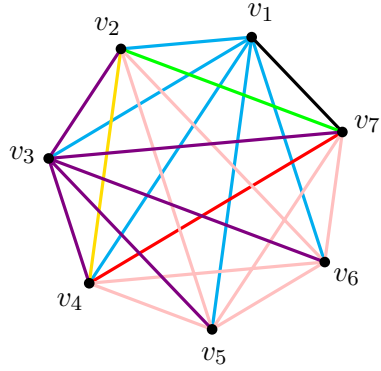
\begin{figure}[htb]
\begin{center}
  \begin{tikzpicture}[  every path/.style={very thick}, 
        every node/.style={circle,fill=black,draw,inner sep = 0.5pt, minimum size= 2pt}
        ]
    \foreach \i in {1,...,7}{
        \path (\i*360/7+18:2) coordinate (n\i);
        \ifnum\i>1\relax
            foreach \j in {\i,...,1}{ (n\i) edge (n\j) } 
        \fi;
    }
    \foreach \i in {1,...,7}{
    \node[
        draw=none,
        fill=none,
        font=\small
    ] at (\i*360/7+18:2.35) {$v_{\i}$};
}
     \draw[cyan] 
    (n1) -- (n2);
    \draw[cyan] 
    (n1) -- (n3);
    \draw[cyan] 
    (n1) -- (n4);
    \draw[cyan] 
    (n1) -- (n5);
     \draw[cyan] 
    (n1) -- (n6);
     \draw[black] 
    (n1) -- (n7);
    \draw[green] 
    (n2) -- (n7);
    \draw[yellow!80!orange] 
    (n2) -- (n4);
     \draw[red] 
    (n4) -- (n7);
    \draw[pink] 
    (n6) -- (n7)(n2) -- (n6)(n4) -- (n6)(n5) -- (n6)(n5) -- (n7)(n4) -- (n5)(n2) -- (n5);
    \draw[violet]
    (n3) -- (n6) (n2) -- (n3)(n3) -- (n7)(n3) -- (n5)(n3) -- (n4);
      ;
    \foreach \i in {1,...,7}{\fill (n\i) circle(2pt);}
  \end{tikzpicture}
  \caption{An obstruction on $K_7$ that has a bicolored vertex at $v_1$ and is non-equinumerous off that vertex, and that is not a minimal obstruction because it contains an absolute rainbow $K_3$ with yellow, green and red edges. }
\label{valvert-odd-bad}
\end{center}
\end{figure}

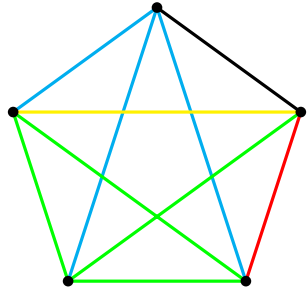
\begin{figure}[htb]

\begin{center}
  \begin{tikzpicture}[  every path/.style={very thick}, 
        every node/.style={circle,fill=black,draw,inner sep = 0pt, minimum size= 2pt}
        ]
      
    \foreach \i in {1,...,5}{
        \path (\i*360/5+18:2) coordinate (n\i);
        \ifnum\i>1\relax
            foreach \j in {\i,...,1}{ (n\i) edge (n\j) } 
        \fi;
    }
   
     \draw[cyan] 
    (n1) -- (n2);
    \draw[cyan] 
    (n1) -- (n3);
    \draw[cyan] 
    (n1) -- (n4);
    \draw[black] 
    (n1) -- (n5);
    \draw[yellow] 
    (n2) -- (n5);
    \draw[red]
    (n4) -- (n5);
    \draw[green]
    (n2) -- (n4)(n2) -- (n3)(n3) -- (n4)(n3) -- (n5);
    \foreach \i in {1,...,5}{\fill (n\i) circle(2pt);}
  \end{tikzpicture}
  \caption{A minimal obstruction on $K_5$ that has both a bicolored vertex and a non-equinumerous coloring of the remaining edges.}
\label{Bivertex-k5}
\end{center}
\end{figure}

When there is a bicolored vertex but the remainder of the coloring is not equinumerous, the graph may (as in Figure \ref{Bivertex-k5}) or may not (as in Figure \ref{valvert-odd-bad}) be a minimal obstruction.  We do not have any criterion for when an arbitrary coloring will or will not be a minimal obstruction.

\begin{question}\label{vvq}
Find an if-and-only-if criterion for a graph in the setting of Theorem \ref{valvert-odd} to be a minimal obstruction.
\end{question}

Next, we consider a graph like that in Theorem \ref{valvert-odd} but for even $n$.  We want the color classes other than those at the special vertex to be nearly equinumerous.
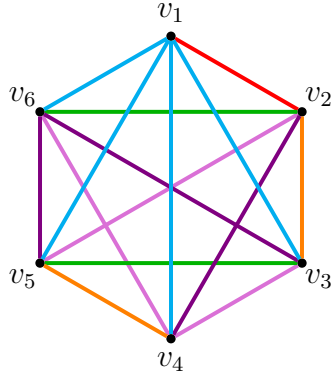
\begin{figure}
    \centering
\[
\begin{tikzpicture}[
    scale=1.0,
    every path/.style={line width=1.5pt},
    every node/.style={
        circle,
        fill=black,
        draw,
        inner sep=0pt,
        minimum size=2pt
    }
]

\foreach \i/\a in {
1/90,
2/30,
3/-30,
4/-90,
5/-150,
6/150}
{
    \node[label=\a:$v_{\i}$] (v\i) at (\a:2) {};
}

\draw[red]
(v1)--(v2);

\draw[Orchid]
(v3)--(v4)
(v2)--(v5)
(v4)--(v6);

\draw[green!70!black]
(v1)--(v4)
(v2)--(v6)
(v3)--(v5);

\draw[orange]

(v2)--(v3)
(v5)--(v4);
\draw[violet]
(v5)--(v6)
(v3)--(v6)
(v2)--(v4);

\draw[cyan]
(v1)--(v3)
(v1)--(v4)
(v1)--(v5)
(v1)--(v6);

\end{tikzpicture}
\]
 \caption{An illustration of Theorem \ref{valvert-even} for $K_6$ with $v_1$ the bicolored vertex.}
    \label{bivertexeven}
\end{figure}
\begin{theorem}[Bicolored Vertex: Even Case]\label{valvert-even}
For even $n\geq3$, let $K_n$ have $n$ color classes such that one vertex $v$ is incident with all the edges of two classes but with no other edges, and the remaining edges are colored in $n-2$ other colors so that each color class has exactly $\frac12(n-2)$ or $\frac12 n$ edges.  Then this graph is a minimal obstruction.
\end{theorem}

\begin{proof}
In outline the proof is similar to that of Theorem \ref{valvert-odd}.  As in that proof, deleting $C_n$ or $C_{n-1}$ or neither, but not both, cannot give an obstruction.  If we delete both, we are left with $K_{n-1}$ having $n-2$ colors; the average size of a color class is $\frac12(n-1)$ so, to obtain near equinumerosity, there should be $\frac12(n-2)$ color classes of $\frac12 n$ edges and equally many of $\frac12(n-2)$ edges.  Since the graph has order $n-1$ and some color classes have more than $\frac12(n-1)$ edges, Suzuki's theorem does not apply, but that of Akbari and Alipour does; hence, the $K_{n-1}$ has a rainbow spanning tree, which is total; thus, $K_n \setminus v$ contains no obstruction.
\end{proof}

We suggest that there are generalizations of Theorems \ref{equinum-odd}--\ref{valvert-even} 
along the following lines, rather like Akbari and Alipour's theorem but without their upper bound on the size of a color class.  Let $K_n$ be colored in $n$ colors with color class sizes $m_1 \geq m_2 \geq \cdots \geq m_n$.  If that sequence is majorized by the sequence $n-1 > n-2 > \cdots > 0$ of vertex-cut sizes of shrinking complete graphs, then this edge-colored complete graph is a minimal obstruction.  More broadly, we suggest there is a generalization to an arbitrary connected graph $G$ with the sequence of vertex-cut sizes replaced by an analogous sequence of cut sizes in $G$ that we cannot precisely define.

\subsection{Bipartition Obstruction}

\begin{construction}[Bipartition obstruction]\label{completed-double-star}
In $K_n$ color a $K_{n-2}$ by a color $c_0$, the $K_2$ on the complementary vertex set by a second color $c_0^*$, and for each $v_1,v_2,\dots,v_{n-2}$ in the former and $w_1,w_2$ in the latter, color edges $v_iw_1$ and $v_iw_2$ by color $c_i$ for a total of $n$ colors.  (This type generalizes the first graph in the second row of Figure \ref{F-allK4}.)

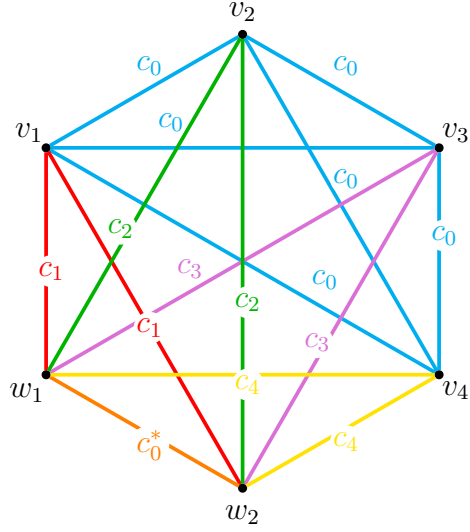
\begin{figure}
    \centering
\begin{tikzpicture}[
    scale= 1.5,
    every path/.style={line width= 1.4 pt},
    every node/.style={
        circle,
        fill=black,
        draw,
        inner sep=0pt,
        minimum size=2pt
    }
]

\tikzset{
    edgelabel/.style={
        midway,
        draw=none,
        fill=white,
        inner sep=1pt,
        font=\scriptsize
    }
}

\foreach \i/\a/\lbl in {
1/90/{v_2},
2/30/{v_3},
3/-30/{v_4},
4/-90/{w_2},
5/-150/{w_1},
6/150/{v_1}}
{
    \node[label=\a:$\lbl$] (v\i) at (\a:2) {};
}


\draw[cyan]
(v1) --
node[
    midway,
    above,
    xshift=2pt,
    yshift=3pt,
    draw=none,
    fill=white
] {$c_0$}
(v2);
\draw[cyan]
(v1) --
node[
    midway,
    above,
    xshift=2pt,
    yshift=3pt,
    draw=none,
    fill=white
] {$c_0$}
(v3);

\draw[cyan]
(v2) --
node[
    pos = 0.7,
    above,
    xshift= 2 pt,
    yshift=3pt,
    draw=none,
    fill=white
] {$c_0$}
(v6);
\draw[cyan]
(v1) --
node[
    midway,
    above,
    xshift=2pt,
    yshift=3pt,
    draw=none,
    fill=white
] {$c_0$}
(v6);
\draw[cyan]
(v2) --
node[
    midway,
    above,
    xshift=2pt,
    yshift=3pt,
    draw=none,
    fill=white
] {$c_0$}
(v3);
\draw[cyan]
(v3) --
node[
    pos = 0.3,
    above,
    xshift=2pt,
    yshift=3pt,
    draw=none,
    fill=white
] {$c_0$}
(v6);

\draw[Orchid]
(v2) --
node[
    pos = 0.65,
    above,
    xshift=2pt,
    yshift=3pt,
    draw=none,
    fill=white
] {$c_3$}
(v5);
\draw[Orchid]
(v2) --
node[
    pos = 0.65,
    above,
    xshift=2pt,
    yshift=3pt,
    draw=none,
    fill=white
] {$c_3$}
(v4);
\draw[red]
(v4) --
node[
    midway,
    below,
    xshift=2pt,
    yshift=3pt,
    draw=none,
    fill=white
] {$c_1$}
(v6);
\draw[red]
(v5) --
node[
    midway,
    below,
    xshift=2pt,
    yshift=3pt,
    draw=none,
    fill=white
] {$c_1$}
(v6);

\draw[green!70!black]
(v1) --
node[
    pos = 0.65,
    above,
    xshift=2pt,
    yshift=3pt,
    draw=none,
    fill=white
] {$c_2$}
(v4);
\draw[green!70!black]
(v1) --
node[
    pos = 0.65,
    above,
    xshift=2pt,
    yshift=3pt,
    draw=none,
    fill=white
] {$c_2$}
(v5);
\draw[orange]
(v4) --
node[
    midway,
    below,
    xshift=2pt,
    yshift=3pt,
    draw=none,
    fill=white
] {$c_{0}^*$}
(v5);
\draw[yellow!85!orange]
(v4) --
node[
    midway,
    below,
    xshift=2pt,
    yshift=3pt,
    draw=none,
    fill=white
] {$c_4$}
(v3);
\draw[yellow!85!orange]
(v5) --
node[
    midway,
    below,
    xshift=2pt,
    yshift=3pt,
    draw=none,
    fill=white
] {$c_4$}
(v3);
\end{tikzpicture}
\caption{Illustration of Construction \ref{completed-double-star} for $n=6$.}
\label{bip-obs}
\end{figure}

\begin{prop}
This is a minimal obstruction if $n \geq 4$.
\end{prop}

\begin{proof}
It is an obstruction because $r_1=n$ and $r_2 = n-1$.  Since removing any one color class reduces $r_1$ while keeping the graph connected, the resulting subgraph with $n-1$ colors is not an obstruction.  

Now consider cases.  Suppose we remove color classes $c_0$ and $c_0^*$.  The remaining graph is $K_{2,n-2}$ in which $r_1 = n-2 = r_2-1$ and each vertex $v_i$, of valency 2, is monochromatic.  Removing any number $s \geq 0$ of additional color classes now reduces both $r_1$ and $r_2$ by $s$, unless $s = n-2$, in which case the remaining graph consists of isolated vertices; in no case do we get an obstruction.  If we remove class $c_0$ and not $c_0^*$ and $s \geq 1$ additional color classes, the same argument applies, except that $s=n-2$ is not different from $s<n-2$.  If we remove class $c_0^*$ and not $c_0$, and $s \geq 1$ additional color classes, then $r_2=n-1$ 
if $s<n-2$ and $n-2$ if $s=n-2$; in each such case $r_1 < r_2$.  Finally, if we keep classes $c_0$ and $c_0^*$ but delete $s \geq 2$ of the 2-edge classes $c_i$, again $r_2 = n-1$ except when $s = n-2$ and in every case $r_1 \leq r_2$ for the deleted subgraph.  Thus, no absolute proper subgraph of this colored $K_n$ is an obstruction, which proves minimality.
\end{proof}

\end{construction}

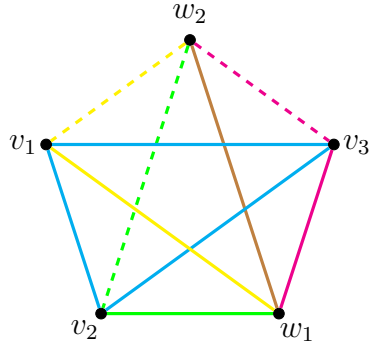
\begin{figure}
\begin{center}
\begin{tikzpicture}[
    every node/.style={
        circle,
        fill=black,
        draw,
        inner sep=0pt,
        minimum size=4pt
    }
]

\node[label=above:$w_2$] (v1) at (90:2) {};
\node[label=right:$v_3$] (v2) at (18:2) {};
\node[label=below right:$w_1$] (v3) at (-54:2) {};
\node[label=below left:$v_2$] (v4) at (-126:2) {};
\node[label=left:$v_1$] (v5) at (162:2) {};

\draw[magenta, dashed, very thick]      (v1)--(v2);
\draw[brown, very thick]     (v1)--(v3);
\draw[green, dashed, very thick]    (v1)--(v4);
\draw[yellow, dashed, very thick]     (v1)--(v5);

\draw[magenta, very thick]  (v2)--(v3);
\draw[cyan, very thick]   (v2)--(v4);
\draw[cyan, very thick]   (v2)--(v5);

\draw[green, very thick]    (v3)--(v4);
\draw[yellow, very thick]   (v3)--(v5);

\draw[cyan, very thick]     (v4)--(v5);

\end{tikzpicture}

\caption{Our counterexample to Question \ref{q-double-star}.  The dashed edges are the deleted edges.}
\label{F-q-double-star}
\end{center}
\end{figure}

\begin{question}\label{q-double-star}
Suppose in Construction \ref{completed-double-star} we delete one edge from some pairs $v_iw_1, v_iw_2$, but not from every such pair.  The result is an obstruction since $r_1 = n > r_2 = n-1$.  Is this obstruction minimal?  
\end{question}

We know it might not be minimal.  For example, if $n=5$ and we delete \emph{all} the edges $v_iw_2$ (see Figure \ref{F-q-double-star}), we get a non-minimal obstruction since what results is the upper right $K_4$ obstruction in Figure \ref{F-allK4} as an absolute subgraph, plus the edge $w_1w_2$.

\subsection{Disconnected minimal obstructions}

A minimal obstruction need not be connected.  We demonstrate this with some examples.

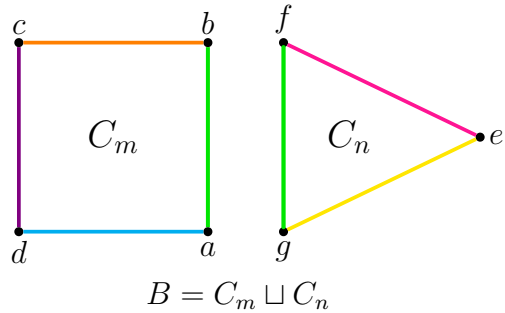
\begin{figure}[hb]
\centering
\begin{tikzpicture}[
    every node/.style={circle,fill=black,draw,inner sep=0pt,minimum size=2pt},
    every path/.style={line width=1.4pt}
]
\begin{scope}
    
\node[label=below:$a$] (A) at (0,0) {};
\node[label=above:$b$] (B) at (0,2.5) {};

\node[label=above:$c$] (D) at (-2.5,2.5) {};
\node[label=below:$d$] (E) at (-2.5,0) {};


\draw[green!90!black,line width= 1.5 pt] (A)--(B);
\draw[orange] (B)--(D);
\draw[violet] (D)--(E);
\draw[cyan] (E)--(A);


\node[draw=none,fill=none,font=\large] at (-1.25,1.25) {$C_m$};

\end{scope}
 \begin{scope}[xshift=1cm]
\node[label=below:$g$] (A) at (0,0) {};
\node[label=above:$f$] (B) at (0,2.5) {};
\node[label=right:$e$] (C) at (2.6,1.25) {};
\draw[
    green!90!black,
    line width= 1.6 pt,
    line cap=round
] (A)--(B);
\draw[DeepPink] (B)--(C);
\draw[yellow!90!orange] (C)--(A);
\node[draw=none,fill=none,font=\large] at (0.87,1.25) {$C_n$};
\end{scope}
 \node[draw=none,fill=none] at (0.4,- 0.85) {$B = C_m \sqcup C_n$};
\end{tikzpicture}
\vspace{-1cm}
\caption{A graph $B$ to illustrate Construction \ref{two-cycles}.}
\label{F-2-cycles}
\end{figure}

\begin{construction}[One shared color]\label{two-cycles}
Let $B$ consist of two disjoint rainbow cycles, $C_m$ and $C_n$, that have one common color.  This is a minimal obstruction with $r_1 = m+n-1$ and $r_2 = m+n-2$.

\begin{figure}[hb]
\centering
\begin{tikzpicture}[
    every path/.style={thick},
    every node/.style={circle,fill=black,draw,inner sep=0pt,minimum size=1pt}
]


\begin{scope}
    \draw[red] (0,0)--(2,0);
    \draw[cyan] (2,0)--(2,2);
    \draw[red] (2,2)--(0,2);
    \draw[cyan] (0,2)--(0,0);
    \draw[green] (0,0)--(2,2);
    \draw[yellow] (0,2)--(2,0);

    \foreach \corner in {(0,0),(2,0),(2,2),(0,2)}
        \node at \corner {};

    \node[
    draw=none,
    fill=none,
    font=\large
] at (1,-0.5) {$B_1$};
\end{scope}

\begin{scope}[xshift=2.8cm]
    \draw[DeepPink] (0,0)--(2,0);
    \draw[teal!80!white] (2,0)--(2,2);
    \draw[DeepPink] (2,2)--(0,2);
    \draw[teal!80!white] (0,2)--(0,0);
    \draw[violet] (0,0)--(2,2);
    \draw[orange] (0,2)--(2,0);

    \foreach \corner in {(0,0),(2,0),(2,2),(0,2)}
        \node at \corner {};

    \node[
    draw=none,
    fill=none,
    font=\large
] at (1,-0.5) {$B_2$};
\end{scope}


\begin{scope}[xshift= 7cm]
    \draw[red] (0,0)--(2,0);
    \draw[cyan] (2,0)--(2,2);
    \draw[red] (2,2)--(0,2);
    \draw[cyan] (0,2)--(0,0);
    \draw[green] (0,0)--(2,2);
    \draw[yellow] (0,2)--(2,0);

    \foreach \corner in {(0,0),(2,0),(2,2),(0,2)}
        \node at \corner {};
\end{scope}

\begin{scope}[xshift= 9.8 cm]
    \draw[DeepPink] (0,0)--(2,0);
    \draw[teal!80!white] (2,0)--(2,2);
    \draw[DeepPink] (2,2)--(0,2);
    \draw[teal!80!white] (0,2)--(0,0);
    \draw[green] (0,0)--(2,2);
    \draw[orange] (0,2)--(2,0);

    \foreach \corner in {(0,0),(2,0),(2,2),(0,2)}
        \node at \corner {};
\end{scope}

\node[
    draw=none,
    fill=none,
    font=\large
] at (9.4,-0.7) {$B = B_1 \sqcup B_2$};
\draw[
    cyan!80!blue,
    solid,
    ->,
    >=Stealth,
    line width=1.5pt
] (5.2,1) -- (6.6,1);
\end{tikzpicture}
\vspace{-1cm}
\caption{Illustration of Construction \ref{disconn-construction}.  The second $B_2$ has one color changed to a color (green) that appears in $B_1$.}
\label{F-2-components}
\end{figure}
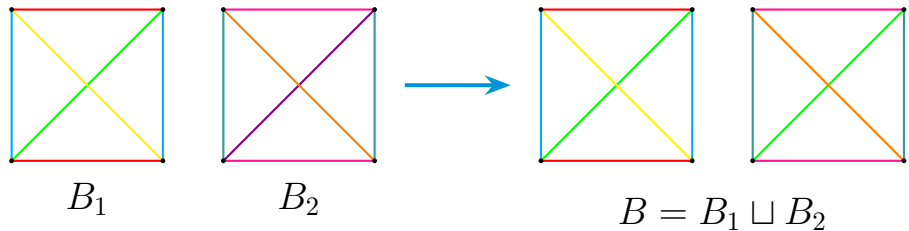

For the proof, consider that deleting the two edges of the common color results in two paths, which by Proposition \ref{forest} cannot be or contain an obstruction.  Deleting one edge of a non-repeated color leaves a disjoint path and a cycle that have one common color, which cannot be an obstruction for the following reason.  Suppose it were.  The path edge with the common color is an isthmus, so its color class contains a bond.  By Lemma \ref{bond-lemma}, deleting that color class must give an obstruction, but it gives three paths, which is a forest and thus not an obstruction, by Proposition \ref{forest}.
\end{construction}

We generalize  Construction \ref{two-cycles} as follows.
\begin{construction}\label{disconn-construction}
Let $B_1$ and $B_2$ be two minimal obstructions having no colors in common.  Form $B$ by taking the disjoint union $B_1 \sqcup B_2$ and changing one color in $B_2$ to a color in $B_1$.  Then $B$ is a minimal obstruction.
\end{construction}

\begin{proof}
$B$ is an obstruction since by applying Proposition \ref{obs-ranks} to the minimality of $B_1$ and $B_2$, we have
$r_1(B) = r_1(B_1)+r_1(B_2) - 1 = r_2(B_1) + r_2(B_2) + 1 = r_2(B) + 1$.

To prove $B$ is minimal, consider deleting one color class $C$.  If it is the ``overlapping'' color class that has edges in both $B_1$ and $B_2$, we have a graph $B \setminus C$ in which, for every absolute subgraph $B'$,  
$$r_1(B') = r_1(B_1 \cap B')+r_1(B_2 \cap B') \leq r_2(B_1 \cap B') + r_2(B_2 \cap B') = r_2(B')$$ 
by minimality of $B_1$ and $B_2$.  Therefore, no subgraph of $B \setminus C$ is an obstruction.

If we delete a non-overlapping color class $C$, say from $B_1$, then $r_1(B \setminus C) = r_1(B) - 1 = r_2(B) = r_2(B_1) + r_2(B_2) = r_2(B_1 \setminus C) + r_2(B_2) = r_2(B \setminus C)$ by Proposition \ref{obs-ranks} applied to $B_1$.  For every absolute subgraph $B'$ of $B \setminus C$, $r_1(B') = r_1(B_1 \cap B') + r_1(B_2) - 1 \leq r_2(B_1 \cap B') + r_2(B_2) = r_2(B')$, so $B'$ is not an obstruction.

We conclude that $B$ is a minimal obstruction.
\end{proof}

The next construction gives a somewhat different disconnected minimal obstruction.

\begin{construction}\label{disc-obs}
Take two separate diamond ($K_4 \setminus e$) graphs (so $r_2=6$) and 7 colors for the 10 edges, so that three colors appear twice and four appear once.  The doubled colors appear at the 3 edges on one trivalent vertex in both diamonds.  The other, ``simple'' colors are applied arbitrarily to the other edges (Figure \ref{F-disconnected}).  
Call this graph $B$.

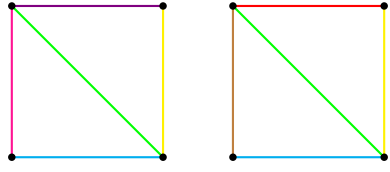
\begin{figure}[htbp]
\begin{center}
  \[\begin{tikzpicture}[
        every path/.style={thick}, 
        every node/.style={circle,fill=black,draw,inner sep = 0pt, minimum size= 2pt}
        ]
    \draw[cyan] (0,0)--(2,0);
    \draw[yellow] (2,0)--(2,2);
    \draw[Purple] (2,2)--(0,2);
    \draw[DeepPink] (0,2)--(0,0);
    \draw[green] (0,2)--(2,0);
    \foreach \corner in {(0,0), (2,0), (2,2), (0,2)} 
        \node at \corner {};
\end{tikzpicture} \quad \quad  \begin{tikzpicture}[
        every path/.style={thick}, 
        every node/.style={circle,fill=black,draw,inner sep = 0pt, minimum size= 2pt}
        ]
    \draw[cyan] (0,0)--(2,0);
    \draw[yellow] (2,0)--(2,2);
    \draw[red] (2,2)--(0,2);
    \draw[brown] (0,2)--(0,0);
    \draw[green] (0,2)--(2,0);
    \foreach \corner in {(0,0), (2,0), (2,2), (0,2)} 
        \node at \corner {};
\end{tikzpicture}\] 
\caption{A disconnected minimal obstruction $B$ as in Construction \ref{disc-obs}.}
\label{F-disconnected}
\end{center}
\end{figure}

To prove minimality, first consider deleting a doubled color from $B$, giving a graph $B'$.  Deleting an edge from each diamond gives either $C_4$ or a triangle with a pendant edge.  If the latter, delete the pendant edge.  The result, call it $B''$, consists of two cycles $C_3$, or two $C_4$'s, or one $C_3$ and one $C_4$.  It is not an obstruction, since $r_1=r_2 = 4$, $6$, or $5$, respectively.  Since $B''$ is not an obstruction, by Lemma \ref{bond-lemma}, $B'$ cannot be an obstruction.  It is easy to see that removing any combination of color classes from $B''$ will keep $r_1 \leq r_2$.

Now, observe that if the first step is to delete a simple color from $B$, let us say from the ``first'' diamond, the result is a graph with $r_1=6=r_2$ and that has an isthmus of a doubled color, say $i$.  Deleting the two $i$-colored edges leaves us in the situation of the preceding paragraph.

It follows that no proper subgraph of $B$ that is a union of color classes is an obstruction; thus, $B$ is minimal.
\end{construction}

Construction \ref{disc-obs} suggests that it and Construction \ref{disconn-construction} have a more complicated generalization that contains both.  We have not tried to formulate such a generalization; we suggest that as a research question.
\begin{problem}
Find a general construction of minimal obstructions that incorporates Constructions \ref{disconn-construction} and \ref{disc-obs}.
\end{problem}

We note that all the results involving disjoint union continue to be valid if the operation is identifying a vertex in each graph instead of disjoint union.  For instance, we might identify the upper right vertex in the left-hand graph of Figure \ref{F-disconnected} with the upper left vertex in the right-hand graph (call the resulting graph $B_1$), and since every subgraph in $B_1$ has the same values of  $r_1$ and $r_2$ as does the corresponding subgraph in the disjoint union $B$, every absolute proper subgraph in $B'$ satisfies $r_1 \leq r_2$; thus $B'$, like $B$, is a minimal obstruction, which has a cut vertex instead of being disconnected.  (The relevant property is that both have multiple blocks.  It is the blocks that determines the matroid rank $r_2$.)

\subsection{Colored-edge sum}

Now we define the colored-edge sum, yet another way to obtain new obstructions.

\begin{construction}  Let $B_1$ and $B_2$ be two minimal obstructions having exactly one color, say $0$, in common.  Choose 0-colored edges $e_1 \in E(B_1)$ and $e_2 \in E(B_2)$.  Identify $e_1$ with $e_2$, forming edge $e$ and an edge-colored graph $B$.  Call $B$ the \emph{colored-edge union} of $B_1$ and $B_2$.
Deleting $C_0$, the color class of $e$, from $B$ gives the \emph{colored-edge sum} of $B_1$ and $B_2$, call it $B'$.
\end{construction}

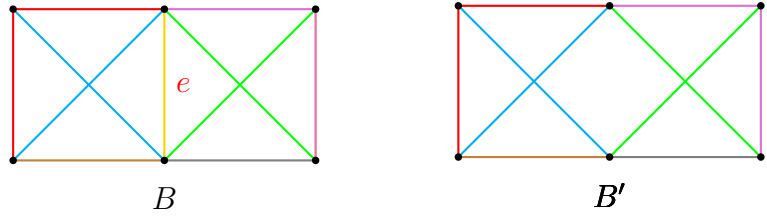
\begin{figure}[htbp]
\begin{center}

\begin{tikzpicture}[
    every path/.style={thick},
    every node/.style={circle,fill=black,draw,inner sep=0pt,minimum size=2pt}
]
    \draw[brown] (0,0)--(2,0);

    \draw[yellow!70!orange] (2,0)--(2,2)
        node[midway,right=3pt,draw=none,fill=none] {$\color{red}e\color{black}$};

    \draw[cyan] (0,2)--(2,0);
    \draw[red] (2,2)--(0,2);
    \draw[red] (0,2)--(0,0);
    \draw[cyan] (0,0)--(2,2);

    \begin{scope}[xshift=2cm]
        \draw[yellow] (2,0)--(2,2);
        \draw[green] (0,2)--(2,0);
        \draw[Grey] (0,0)--(2,0);
        \draw[Orchid] (2,0)--(2,2);
        \draw[Orchid] (2,2)--(0,2);
        \draw[green] (0,0)--(2,2);
    \end{scope}

    \foreach \corner in {(0,0), (2,0), (2,2), (0,2), (4,0), (4,2)} {
        \node at \corner {};
    }

    \node[draw=none,fill=none] at (2,-0.5) {$B$};

\end{tikzpicture}
\qquad\qquad
\begin{tikzpicture}[
        every path/.style={thick}, 
        every node/.style={circle,fill=black,draw,inner sep = 0pt, minimum size= 2pt}
        ]
    \draw[brown] (0,0)--(2,0);
    \draw[cyan] (0,2)--(2,0);
    \draw[red] (2,2)--(0,2);
    \draw[red] (0,2)--(0,0);
    \draw[cyan] (0,0)--(2,2);
    \begin{scope}[xshift=2cm]
            \draw[green] (0,2)--(2,0);
          \draw[Grey] (0,0)--(2,0);
          \draw[Orchid] (2,0)--(2,2);
          \draw[Orchid] (2,2)--(0,2);
          \draw[green] (0,0)--(2,2);
    \end{scope}
    
    \foreach \corner in {(0,0), (2,0), (2,2), (0,2), (4,0), (4,2)} {
        \node at \corner {};

    \node[draw=none,fill=none] at (2,-0.5) {$B'$};
    }
\end{tikzpicture}
\caption{A colored-edge union of two colored $K_4$ obstructions, along the central edge (left), and its corresponding colored-edge sum (right) with a minimal obstruction $B_0 = B'$. 
}
\label{edge-union-sum}
\end{center}
\end{figure}

\begin{prop}\label{edge-union}
Assume that each of $B_1 \setminus C_0$ and $B_2 \setminus C_0$ is connected.  
Then the colored-edge sum $B'$ of the minimal obstructions $B_1$ and $B_2$ contains an absolute subgraph that is a minimal obstruction $B_0$ and is not a subgraph of either $B_1$ or $B_2$.
\end{prop}

\begin{proof}
Consider maximal forests $T_1 \subseteq B_1$ containing $e_1$, $T_2 \subseteq B_2$ containing $e_2$, and $T = T_1 \cup T_2$, which is a maximal forest in $B$ that contains $e$.  Thus, $r_2(B) = r_2(B_1) + r_2(B_2) - 1$.  
Because each $B_i$ is a minimal obstruction, we have $r_1(B_i) = r_2(B_i)+1$.  
Now we compute $r_1(B),$
\begin{equation}
r_1(B) = r_1(B_1)+r_1(B_2) - 1 = r_2(B_1) + r_2(B_2) + 1 = r_2(B) + 2. 
\label{E-edge-union}
\end{equation}

Delete the color class $C_0$ from $B$, giving the colored-edge sum $B'$.  Then $r_1(B') = r_1(B) - 1 = r_2(B) + 1$ by Equation \eqref{E-edge-union}. 
Since $B' \subset B$, we have $r_1(B') > r_2(B) \geq r_2(B')$, so $B'$ is an obstruction.

Since $B'$ is an obstruction, it contains a minimal obstruction $B_0$ as an absolute subgraph. 
If $B_0$ were a subgraph of $B_1$ or $B_2$, it would be an absolute proper subgraph (being an absolute subgraph of $B'$, and because $B_1'$ and $B_2'$ are absolute subgraphs of $B'$) and therefore not an obstruction (by the minimality of $B_1$ and $B_2$).  Consequently, $B_0 \not\subseteq B_1, B_2$.  That concludes the proof.
\end{proof}

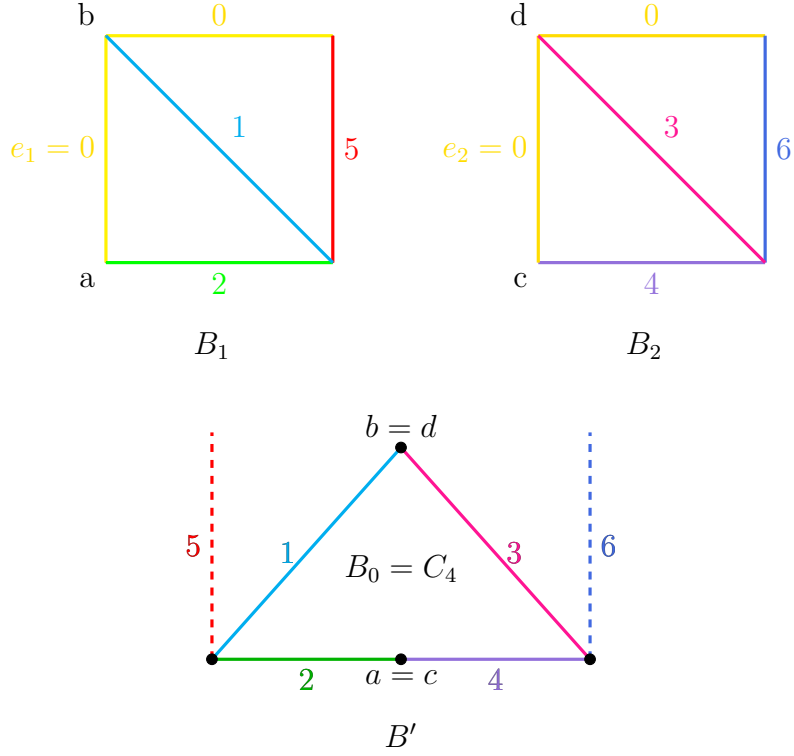
\begin{figure}[ht]
\begin{center}
\[
\begin{tikzpicture}
\coordinate (A) at (-1.5,-1.5);
\coordinate (B) at (-1.5, 1.5);
\coordinate (C) at ( 1.5, 1.5);
\coordinate (D) at ( 1.5,-1.5);
\draw (A) rectangle (C);
\draw[yellow, very thick] (B) -- (C); 
\draw[yellow, very thick] (A) -- (B); 
\draw[green, very thick] (A) -- (D); 
\draw[red, very thick]    (C) -- (D); 
\draw[cyan, very thick]
    (B) -- (D)
    node[midway, above right] {1};

\node[above left] at (B) {b};
\node[below left] at (A) {a};
\node[above, text=yellow!80!orange] at (0,1.5) {$0$};
\node[right, text=red]    at (1.5,0) {5};
\node[below, text = green]   at (0,-1.5) {2};
\node[left, text=yellow!80!orange]  at (-1.5,0) {$e_1 = 0$};
\node[draw=none,fill=none] at (-0.1,-2.6) {$B_1$};
\end{tikzpicture}\quad\quad \begin{tikzpicture}
\coordinate (A) at (-1.5,-1.5);
\coordinate (B) at (-1.5, 1.5);
\coordinate (C) at ( 1.5, 1.5);
\coordinate (D) at ( 1.5,-1.5);
\draw[yellow!80!orange, very thick]       (B) -- (C); 
\draw[RoyalBlue, very thick]       (C) -- (D); 
\draw[MediumPurple, very thick] (A) -- (D); 
\draw[yellow!80!orange, very thick]       (A) -- (B); 

\draw[DeepPink, very thick]
    (B) -- (D)
    node[midway, above right] {3};
\node[above left] at (B) {d};
\node[below left] at (A) {c};

\node[above, text=yellow!80!orange]     at (0,1.5) {$0$};
\node[right, text= RoyalBlue]         at (1.5,0) {6};
\node[below, text=MediumPurple]        at (0,-1.5) {4};
\node[left, text= yellow!80!orange]               at (-1.5,0) {$e_2 = 0$};
\node[draw=none,fill=none] at (-0.1,-2.6) {$B_2$};
\end{tikzpicture}\]
\[
\begin{tikzpicture}

\coordinate (A) at (0,0);
\coordinate (B) at (5,0);
\coordinate (C) at (2.5,2.8);
\coordinate (M) at ($(A)!0.5!(B)$);

\draw[green!70!black, very thick] (A) -- (M); 
\draw[MediumPurple, very thick]   (M) -- (B); 

\draw[cyan, very thick]      (A) -- (C); 
\draw[DeepPink, very thick]  (B) -- (C); 

\draw[red, very thick, dashed]       (A) -- ($(A)+(0,3)$); 
\draw[RoyalBlue, very thick, dashed] (B) -- ($(B)+(0,3)$); 

\filldraw[black] (A) circle (2pt);
\filldraw[black] (B) circle (2pt);
\filldraw[black] (C) circle (2pt);
\filldraw[black] (M) circle (2pt);

\node[above] at (C) {$b=d$};
\node[below] at (M) {$a=c$};

\node[below] at ($(A)!0.5!(M)$) {$2$};
\node[below] at ($(M)!0.5!(B)$) {$4$};

\node[left]  at ($(A)!0.5!(C)$) {$1$};
\node[right] at ($(B)!0.5!(C)$) {$3$};

\node[left]  at ($(A)+(0,1.5)$) {$5$};
\node[right] at ($(B)+(0,1.5)$) {$6$};
\node[below, text=green!70!black]
    at ($(A)!0.5!(M)$) {$2$};

\node[below, text=MediumPurple]
    at ($(M)!0.5!(B)$) {$4$};

\node[left, text=cyan]
    at ($(A)!0.5!(C)$) {$1$};

\node[right, text=DeepPink]
    at ($(B)!0.5!(C)$) {$3$};

\node[left, text=red]
    at ($(A)+(0,1.5)$) {$5$};

\node[right, text=RoyalBlue]
    at ($(B)+(0,1.5)$) {$6$};
\node[draw=none, fill=none] at (2.5,1.2) {$B_0 = C_4$};
 \node[draw=none,fill=none] at (2.5,- 1.0) {$B'$};

\end{tikzpicture}
\]

\caption{The bottom graph $B'$ is the colored-edge sum of the two top graphs.  It is a non-minimal obstruction with the minimal obstruction $B_0$ (a rainbow $C_4$) shown as the solid edges. }
\label{nonmin-edge-sum}
\end{center}
\end{figure}

\begin{example}\label{nonmin-edge-sum-example}
The minimal obstruction in the whole colored-edge sum $B'$ can be $B'$, as in Figure \ref{edge-union-sum}, but it seems likely it is usually a proper subgraph, as in Figure \ref{nonmin-edge-sum}.  This example is built from two copies of a diamond-graph minimal obstruction (Figure \ref{F-alldiamond}), with an edge in the quadrilateral as the identified edge $e$.  The significant thing about the example is that the sum results in an absolute proper subgraph  that is a rainbow $C_4$ that is not contained in either of the two original diamonds.
\end{example}

\section{Acknowledgments}
The first author would like to deeply thank Tarik Aougab, Gary Gordon, Kennard Lee, Piotr Przytycki, Johannes Rau, Roberta Shapiro, Daniel Studenmund, Edward Swartz and Zhaoshen Zhai for the useful and fruitful discussions about this problem and its accompanying geometric topological problem. The first author also extends deep thanks and heartfelt gratitude to Paul Apisa, David Futer, Mathew Stover, and Bena Tshishiku for giving her the opportunity to present lightning talks at CRM-Montreal, Brown University, University of Wisconsin- Madison, and Rutgers University on this research problem and its accompanying geometric topological one. The first author is so grateful and would like to express her deepest gratitude to Matt Bainbridge for giving her the opportunity to present a poster at Bloomington Geometry Workshop 2026 at Indiana University on this problem and its accompanying geometric topological one and lastly, the first author would like to express her heartfelt gratitude for Carnegie Mellon university and the organizers of the conference named Combinatorics at the Confluence for allowing her to present a poster on this paper at this conference. 



\begin{thebibliography}{99}

\bibitem{AA} S. Akbari and A. Alipour, Multicolored trees in complete graphs, \emph{J.\ Graph Theory} 54 (2007), no.\ 3, 221--232. 

\bibitem{AG} Tarik Aougab and Jonah Gaster, From arcs to curves: quadratic growth of 1-systems, submitted.  arXiv:2508.05555.

\bibitem{Bro} Hajo Broersma and Xueliang Li, Spanning trees with many or few colors in edge-colored graphs.  \emph{Discuss. Math. Graph Theory} 17 (1997), no. 2, 259--269.

\bibitem{Bru} Richard A. Brualdi and Susan Hollingsworth, Multicolored trees in complete graphs, \emph{J.\ Combin.\ Theory Ser.\ B} 68 (1996), 310--313.

\bibitem{GK} Stefan Glock, Daniela K\"uhn, Richard Montgomery, and Deryk Osthus, Decompositions into isomorphic rainbow spanning
trees, \emph{J.\ Combin.\ Theory Ser.\ B} 146 (2021), 439--484.

\bibitem{GM} Gary Gordon and Jennifer McNulty, \emph{Matroids:
A Geometric Introduction}, Cambridge University Press, Cambridge, 2012.

\bibitem{HP} Frank Harary and Edward M.\ Palmer, \emph{Graphical Enumeration}, Academic Press, New York, 1973.

\bibitem{Hu} Tony Huynh, Answer to ``Total rainbow'' trees,  MathOverflow, May 2, 2020, \url{https://mathoverflow.net/questions/359185/total-rainbow-trees}.

\bibitem{JMM} M. Juvan, A. Malni\v{c}, and B. Mohar, Systems of curves on surfaces, \emph{J. Combin. Theory Ser. B} 68 (1996), no. 1, 7--22.

\bibitem{KL} Mikio Kano and Xueliang Li, Monochromatic and heterochromatic subgraphs in edge-colored graphs - a survey, \emph{Graphs Combin.}\ 24 (2008), 237--263.

\bibitem{MRT} Justin Malestein, Igor Rivin, and Louis Theran, Topological designs, \emph{Geom. Dedicata} 168 (2014), 221--233.

\bibitem{OEIS} OEIS Foundation, The On-Line Encyclopedia of Integer Sequences, 2026, online at https://oeis.org

\bibitem{Ox} James Oxley, \emph{Matroid Theory}, 2nd ed., Oxford Univ.\ Press, Oxford, 2011.  See Theorem 11.3.15.

\bibitem{Suz} Kazuhiro Suzuki, A necessary and sufficient condition for the existence of a heterochromatic spanning tree in a graph, \emph{Graphs Combin.} 22 (2006), 261--269.

\bibitem{Vondrak} Jan Vondr\'ak, Lecture notes for ``Polyhedral Techniques in Combinatorial Optimization'', lecture 17, 2017.

\end{thebibliography}
\end{document}